\documentclass[11pt,twoside]{article}
\usepackage{authblk}
\usepackage[scaled=0.8]{DejaVuSansMono}
\usepackage{emptypage}
\usepackage[a4paper,marginratio={5:5,5:7},width=147mm]{geometry}
\usepackage[T2A,T1]{fontenc}
\usepackage[utf8x]{inputenc}

\usepackage{fancyhdr}
\newcommand{\maybebf}{}

\fancypagestyle{plain}{%
    \fancyhf{}%
    \fancyfoot[C]{\maybebf \thepage}
    
}
\fancypagestyle{intro}{
    \renewcommand{\sectionmark}[1]{\markboth{##1}{}}
    \fancyhf{}
    \fancyhead[LE,RO]{\leftmark}
    \fancyfoot[C]{\maybebf \thepage}

}

\renewcommand{\sectionmark}[1]{%
  \markboth{%
    \ifnum\value{section}>0
      \maybebf{\thesection.}\space
    \fi
    #1%
  }{%
    \ifnum\value{subsection}=0
            \thesection. #1
        \fi%
  }%
}
\renewcommand{\subsectionmark}[1]{%
    \markright{%
        \ifnum\value{subsection}>0
            \thesubsection. #1
        \fi%
    }%
}
\usepackage[english]{babel}
\usepackage{graphicx}
\usepackage{amsmath}
\usepackage{amssymb}
\usepackage{amsthm}
\usepackage{mathtools}
\usepackage{microtype}
\usepackage{tikz}
\usetikzlibrary{arrows.meta,arrows,decorations.pathmorphing,shapes,knots,hobby,quotes,positioning}
\usepackage{asymptote}
\usepackage{tikz-cd}
\usepackage{parskip}
\usepackage{float}
\usepackage{caption}
\usepackage{subcaption}
\usepackage{enumitem}
\usepackage{adjustbox}
\usepackage{xfakebold}
\usepackage{mdframed}
\usepackage{booktabs}
\usepackage{multirow}
\usepackage{siunitx}
\usepackage{nicematrix}
\NiceMatrixOptions{xdots/shorten=0.87em, renew-dots}

\usepackage[calc,showdow,english]{datetime2}
\DTMnewdatestyle{mydateformat}{%
}

\usepackage{comment}
\usepackage{import}
\usepackage{standalone}
\usepackage{xcolor}
\usepackage{csquotes}
\usepackage[doi=false,isbn=false,
maxbibnames=99,backend=biber,style=numeric,sorting=nyt]{biblatex}
\usepackage[anythingbreaks]{breakurl}
\usepackage{hyperref}
\definecolor{hlinkcol}{HTML}{A00000}
\definecolor{hcitecol}{HTML}{308030}
\hypersetup{
    colorlinks,
    linkcolor={teal!100!black},
    citecolor={green!50!black},
    urlcolor={blue!0!black}
}
\definecolor{zg}{gray}{0.6}
\newcommand{\zg}[1]{\textcolor{zg}{#1}}

\theoremstyle{plain}
\newtheorem{thm}{Theorem}[section]
\newtheorem{prop}[thm]{Proposition}
\newtheorem{lemma}[thm]{Lemma}
\newtheorem{cor}[thm]{Corollary}

\newtheorem*{thm*}{Theorem}
\newtheorem*{prop*}{Proposition}
\newtheorem*{conj*}{Conjecture}
\newtheorem*{cor*}{Corollary}
\newtheorem*{thmvp*}{Theorem \ref{thm:virt-par}}

\theoremstyle{definition}
\newtheorem{defin}[thm]{Definition}

\newtheorem{question}[thm]{Question}
\newtheorem*{defin*}{Definition}

\theoremstyle{remark}
\newtheorem{remark}[thm]{Remark}

\counterwithin{figure}{section}
\counterwithin{table}{section}

\newcommand{\id}{\mathrm{id}}
\newcommand{\ZZ}{\mathbb Z}
\newcommand{\QQ}{\mathbb Q}
\newcommand{\RR}{\mathbb R}

\newcommand{\HH}{\mathbb H}

\newcommand{\commentt}[1]{}

\tikzset{
dot/.style = {circle, fill, minimum size=#1,
              inner sep=0pt, outer sep=0pt},
dot/.default = 5pt
}

\DeclareMathOperator{\image}{im}

\DeclareMathOperator{\Hom}{Hom}

\DeclareMathOperator{\Fac}{Fac}

\newcommand{\fbseries}{
\unskip\setBold\aftergroup\unsetBold\aftergroup\ignorespaces}
\makeatletter
\newcommand{\setBoldness}[1]{\def\fake@bold{#1}}
\makeatother

\tikzset{
    rots/.style={anchor=south, rotate=90, inner sep=.5mm}
}

\newcommand\restr[2]{{
  \left.\kern-\nulldelimiterspace 
  #1 
  \vphantom{\big|} 
  \right|_{#2} 
  }}

\author{Jacopo G. Chen\thanks{Scuola Normale Superiore, Pisa, Italy. Email: \href{mailto:jacopo.chen@sns.it}{\texttt{jacopo.chen@sns.it}}}}
\title{Some closed hyperbolic $5$-manifolds}

\begin{document}

\DTMsetdatestyle{mydateformat}
\date{}
\maketitle

\begingroup
\centering\small
\textbf{Abstract}\par\smallskip
\begin{minipage}{\dimexpr\paperwidth-9.5cm}
We give an explicit construction of a family of closed arithmetic hyperbolic $5$-manifolds, tessellated by $\text{\num{117964800}} = 512 \cdot 16 \cdot 14400$ copies of a Coxeter simplicial prism. We proceed to study various properties of these manifolds, such as the volume and the first Betti number. We also describe a related family of $5$-manifolds with a larger volume, but a simpler construction.
\end{minipage}
\par\endgroup
\section*{Introduction}
    
It is well known that hyperbolic manifolds of finite volume, both compact and non-compact, exist in every dimension $n \ge 2$, as a consequence of the Borel--Harish-Chandra theorem. An interesting problem arises when we ask to what extent these manifolds can be described \emph{explicitly}; by this we mean some sort of combinatorial description, be it a CW complex structure or a gluing procedure for elementary geometric pieces.
These descriptions are often conducive to the computation of various topological invariants, such as the Betti numbers.

Prime candidates for the gluing of elementary pieces are hyperbolic Coxeter polytopes, that is, finite-volume quotients of hyperbolic space by Coxeter groups. This strategy is especially fruitful in the cusped case: see for example the family of right-angled hyperbolic polytopes $P_n$~\cite[Section~3]{PV}, which can be glued via the \emph{coloring method} (see below). By contrast, the compact case appears to be more complex. Already in dimension $4$, the smallest known manifolds~\cite{conder-maclachlan,long} are tessellated by \num{115200} copies of a Coxeter simplex, and in dimension $5$, there are no known explicit descriptions of any closed hyperbolic manifold. In particular, there is no closed hyperbolic $5$-manifold $M$ for which $b_1(M)$ is known.

In this paper, we construct a family $\{N_{i,I}^\pm\}$ of closed orientable hyperbolic $5$-manifolds having vanishing first Betti number, tessellated by $\text{\num{117964800}} = 512 \cdot 16 \cdot 14400$ copies of an arithmetic Coxeter polytope $P$, with Coxeter diagram
\[
    \raisebox{-0.4\height}{
    \begin{tikzpicture}
        \node[dot, fill=black] (a) at (0,0) {};
        \node[dot, fill=black] (b) at (1,0) {};
        \node[dot, fill=black] (c) at (2,0) {};
        \node[dot, fill=black] (d) at (3,0) {};
        \node[dot, fill=black] (e) at (4,0) {};
        \node[dot, fill=black] (f) at (5,0) {};
        \node[dot, fill=black] (g) at (6,0) {};
        \node[] (x) at (3,0.25) {};
        \node[] (x) at (3,-0.25) {};
        \draw[draw=black, double distance=3pt, thin, solid] (a) -- (b);
        \draw[draw=black, thin, solid] (a) -- (b);
        \draw[draw=black, thin, solid] (b) -- (c);
        \draw[draw=black, thin, solid] (c) -- (d);
        \draw[draw=black, thin, solid] (d) -- (e);
        \draw[draw=black, double distance=2.5pt, thin, solid] (e) -- (f);
        \draw[draw=black, thin, dashed] (f) -- (g);
    \end{tikzpicture}
    }
\]
A notable property of $P$, which is combinatorially a prism on a $4$-simplex, is that $14400$ copies of it can be joined to make a $120$-cell prism $Q$ (Figure~\ref{fig:prisms-PQ}), whose two bases have dihedral angles of $2\pi/3$ and $\pi/2$. Moreover, the sides of $Q$ make angles of $\pi/2$ and $\pi/4$ with the bases, respectively.

Our main strategy extends the classical \emph{coloring method}, which constructs a \emph{real toric manifold} from a right-angled polytope $P$ by coloring its facets with binary vectors (say, in $\ZZ_2^s$) and gluing $2^s$ copies of $P$. In dimension $5$ there are no compact right-angled polytopes, so we consider a generalization: manifolds with right-angled corners.

We construct such manifolds by gluing copies of $Q$ onto a suitable $4$-manifold tessellated by $120$-cells (Figure~\ref{fig:mfd-corners}). This process can be done in such a way as to obtain an orientable right-angled $5$-manifold, with many facets (codimension-$1$ boundary components) isometric to a hyperbolic right-angled $120$-cell.
We then apply the coloring method to these right-angled manifolds; a careful choice of coloring, based on the theory of \emph{linear binary codes}, allows us to substantially reduce the size of the final manifolds.

In more detail, our construction is based on some hyperbolic $4$-manifolds $Z_i$, $i=1,\dots,8$, of Euler characteristic $8$, found by Long~\cite{long}. Their $17$-fold Galois coverings $X_i$ can then be used to construct orientable right-angled $5$-manifolds $Y_i$, which can undergo the coloring construction. The definition of the Long manifolds $Z_i$ is algebraic in nature, which simplifies the study of the adjacency graph of the facets of $Y_i$ (which turns out to be the same for all $i$). With the help of the computer algebra system SageMath~\cite{sagemath}, we single out and classify some symmetrical $17$-colorings of this graph.

This already gives many real toric manifolds (which we call $M_{i,I}$) composed of $2^{17}$ copies of $Y_i$. These colorings can be converted into more efficient colorings, with binary vectors in $\ZZ_2^9$, using the properties of an exceptionally symmetrical $9$-dimensional subspace of $\ZZ_2^{17}$, called a \emph{quadratic residue code}. We call the resulting manifolds $\widehat M_{i,I}^\pm$.

Finally, these $5$-manifolds can be quotiented by a $17$-fold symmetry, resulting in relatively small manifolds $N_{i,I}^\pm$, tessellated by $\num{117 964 800}$ copies of $P$, and with hyperbolic volume less than $\num{250000}$ (computed exactly in Section~\ref{sec:properties}).
The final manifolds can be classified by computer and fall into \num{1600432} isometry classes. In summary, we have:

\begin{thm*}
    There exist some closed orientable hyperbolic $5$-manifolds having $b_1 = 0$ and volume $<\num{250000}$.
\end{thm*}

In Section~\ref{sec:properties}, we study various properties of the $5$-manifolds: in particular, we show that they are orientable, and we compute their volume from a formula for $\operatorname{vol}(P)$~\cite{vol-p}.
We also compute the first Betti number for $N_{i,I}^\pm$ and $\widehat{M}_{i,I}^\pm$ as follows. In the case of manifolds obtained by coloring a right-angled polytope, the Betti numbers can be obtained through a combinatorial formula of Choi--Park~\cite{choi-park}. While the pieces $Y_i$ have nontrivial topology, they can be made contractible by collapsing their embedded copy of $X_i$ to a point, which is enough to apply the formula; the collapsing operation does not change the first Betti number, since $b_1(X_i) = 0$. Using SageMath and a custom program, we find that all the manifolds $N_{i,I}^\pm$ and $\widehat{M}_{i,I}^\pm$ have vanishing first Betti number. We note that computing higher Betti numbers is much harder (Remark~\ref{rmk:higher-betti}) and remains an open question.

We also discuss parallelizability of manifolds obtained with this method through the following general result, which may be of independent interest:
\begin{thmvp*}
    Let $n \in \{1,3,5,7\} \cup \{4k+1 \mid k \ge 2\}$. Then every closed hyperbolic $n$-manifold $M$ is virtually parallelizable.
\end{thmvp*}

Lastly, one may notice that the definition of the manifolds $N_{i,I}^\pm$ involves many non-canonical choices, starting from the initial Long manifold $Z_i$. In Section~\ref{sec:niceX}, we introduce another $4$-manifold $X$, tessellated by $650$ $120$-cells, as opposed to $136$ for the $X_i$; its definition, based on~\cite{everitt-maclachlan}, is arguably much more elegant. We also construct the adjacency graph of the resulting right-angled $5$-manifold $Y$ and study a few possible colorings.

The code used to carry out the computations in this paper can be found in a GitHub repository~\cite{git}.
\subsection*{Structure of the paper}
In Section~\ref{sec:prism}, we introduce the Coxeter polytope $P$. Then, in Section~\ref{sec:4mfds}, we define the Long manifolds $Z_i$ and use them to construct right-angled $5$-manifolds $Y_i$. In Section~\ref{sec:coloring}, we define some colorings on the facets of $Y_i$ in order to construct small $5$-manifolds. Then, in Section~\ref{sec:properties}, we study some properties of the $5$-manifolds constructed previously, including first homology and virtual parallelizability.
Finally, in Section~\ref{sec:niceX}, we define the $4$-manifold $X$, the resulting right-angled $5$-manifold $Y$, and some colorings.
\subsection*{Acknowledgments}
I am grateful to my advisor Bruno Martelli for his support and guidance during the writing of this paper.
\section{A hyperbolic Coxeter prism}\label{sec:prism}
A well-known general method to obtain new hyperbolic manifolds consists of attaching primitive \emph{pieces}, usually Coxeter polytopes, along their boundary, in such a way that they close up without singularities.

Let us introduce a hyperbolic Coxeter $5$-polytope $P$, corresponding to the Coxeter group
\begin{equation}
     \Gamma \coloneqq \raisebox{-0.6\height}{
    \begin{tikzpicture}
        \node[dot, fill=black, label={below:\strut $a$}] (a) at (0,0) {};
        \node[dot, fill=black, label={below:\strut $b$}] (b) at (1,0) {};
        \node[dot, fill=black, label={below:\strut $c$}] (c) at (2,0) {};
        \node[dot, fill=black, label={below:\strut $d$}] (d) at (3,0) {};
        \node[dot, fill=black, label={below:\strut $e$}] (e) at (4,0) {};
        \node[dot, fill=black, label={below:\strut $f$}] (f) at (5,0) {};
        \node[dot, fill=black, label={below:\strut $g$}] (g) at (6,0) {};
        \node[] (x) at (3,0.25) {};
        \node[] (x) at (3,-0.25) {};
        \draw[draw=black, double distance=3pt, thin, solid] (a) -- (b);
        \draw[draw=black, thin, solid] (a) -- (b);
        \draw[draw=black, thin, solid] (b) -- (c);
        \draw[draw=black, thin, solid] (c) -- (d);
        \draw[draw=black, thin, solid] (d) -- (e);
        \draw[draw=black, double distance=2.5pt, thin, solid] (e) -- (f);
        \draw[draw=black, thin, dashed] (f) -- (g);
    \end{tikzpicture}
    }
    \label{eq:P-labeling}
\end{equation}
This polytope was first studied by Bugaenko~\cite{bugaenko}, and it is remarkably simple: its number of facets exceeds its dimension by $2$ (as opposed to $1$ for simplices).
In fact, $P$ is a prism with bases $f$ and $g$, which are combinatorially isomorphic to the fundamental domain of the subgroup $\langle a,b,c,d,e \rangle$ acting on $\HH^4$; what is more, both bases are Coxeter $4$-simplices.


The facet $g$ makes right angles with all the other facets it intersects; hence, it is isometric to the Coxeter $4$-simplex $[5,3,3,3]$, which is the fundamental orthoscheme of a hyperbolic $120$-cell with dihedral angles of $2\pi / 3$ (of order $3$), and is described by the subdiagram 
\begin{equation}
\begin{tikzpicture}
    \node[dot, fill=black] (a) at (0,0) {};
    \node[dot, fill=black] (b) at (1,0) {};
    \node[dot, fill=black] (c) at (2,0) {};
    \node[dot, fill=black] (d) at (3,0) {};
    \node[dot, fill=black] (e) at (4,0) {};
    \draw[draw=black, double distance=3pt, thin, solid] (a) -- (b);
    \draw[draw=black, thin, solid] (a) -- (b);
    \draw[draw=black, thin, solid] (b) -- (c);
    \draw[draw=black, thin, solid] (c) -- (d);
    \draw[draw=black, thin, solid] (d) -- (e);
\end{tikzpicture}
\end{equation}
induced by $\{a,b,c,d,e\}$.
The situation for $f$ is slightly more complex: to compute the dihedral angle between $d \cap f$ and $e \cap f$ as facets of $f$, we note that the link of $d\cap e\cap f$ is a spherical triangle with angles of $\pi/3, \pi/2, \pi/4$. By spherical trigonometry, the angle between $d \cap f$ and $e \cap f$ is $\pi/4$. Hence, the facet $f$ is the simplex $[5,3,3,4]$, i.e., the fundamental orthoscheme of a right-angled (order-$4$) $120$-cell:
\begin{equation}
    \begin{tikzpicture}
    \node[dot, fill=black] (a) at (0,0) {};
    \node[dot, fill=black] (b) at (1,0) {};
    \node[dot, fill=black] (c) at (2,0) {};
    \node[dot, fill=black] (d) at (3,0) {};
    \node[dot, fill=black] (e) at (4,0) {};
    \draw[draw=black, double distance=3pt, thin, solid] (a) -- (b);
    \draw[draw=black, thin, solid] (a) -- (b);
    \draw[draw=black, thin, solid] (b) -- (c);
    \draw[draw=black, thin, solid] (c) -- (d);
    \draw[draw=black, double distance=2.5pt, thin, solid] (d) -- (e);
\end{tikzpicture}
\end{equation}
The orbit of $P$ with respect to the subgroup $\langle a,b,c,d \rangle$ is a $120$-cell prism $Q$, which realizes a \emph{cobordism} between two hyperbolic $120$-cells with different dihedral angles (Figure~\ref{fig:prisms-PQ}). The sides of $Q$ make angles of $\pi/2$ with the order-$3$ base and $\pi/4$ with the order-$4$ base.

\begin{figure}[ht]
    \centering
    \includegraphics[width=0.55\textwidth]{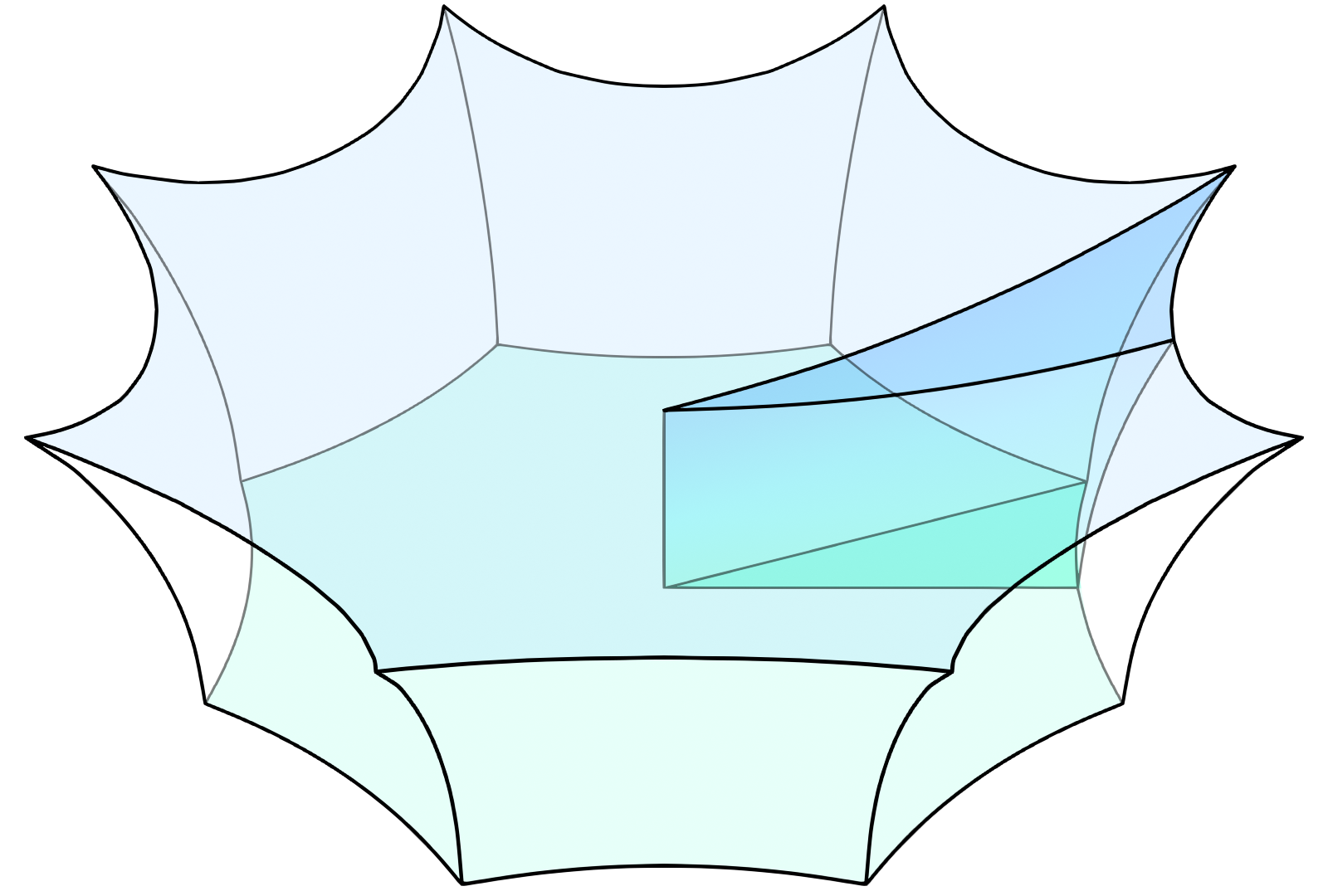}
    \caption{A schematic three-dimensional representation of the simplicial prism $P$ within the $120$-cell prism $Q$. 
    }
    \label{fig:prisms-PQ}
\end{figure}

The order-$3$ base sits naturally inside a $120$-cell honeycomb which tiles a copy of $\HH^4$. By placing copies of $Q$ onto $\HH^4$ in the same pattern, we obtain a combinatorially infinite polytope, where one facet is $\HH^4$, and the rest are infinitely many right-angled $120$-cells meeting at right angles with each other. The vertices thus formed have the geometry of a $5$-cube vertex.

While this object is infinite, it is still possible to find compact quotients of the order-$3$ $120$-cell honeycomb. We can glue copies of $Q$ onto such a manifold, obtaining an orientable $5$-manifold with corners, with right angles between its codimension-$1$ boundary components (Figure~\ref{fig:mfd-corners}). A slight generalization of the \emph{coloring method}, usually applied to right-angled polytopes, leads to a closed orientable manifold, as we will show in the following sections.

\begin{figure}[ht]
    \centering
    \includegraphics[width=0.55\textwidth]{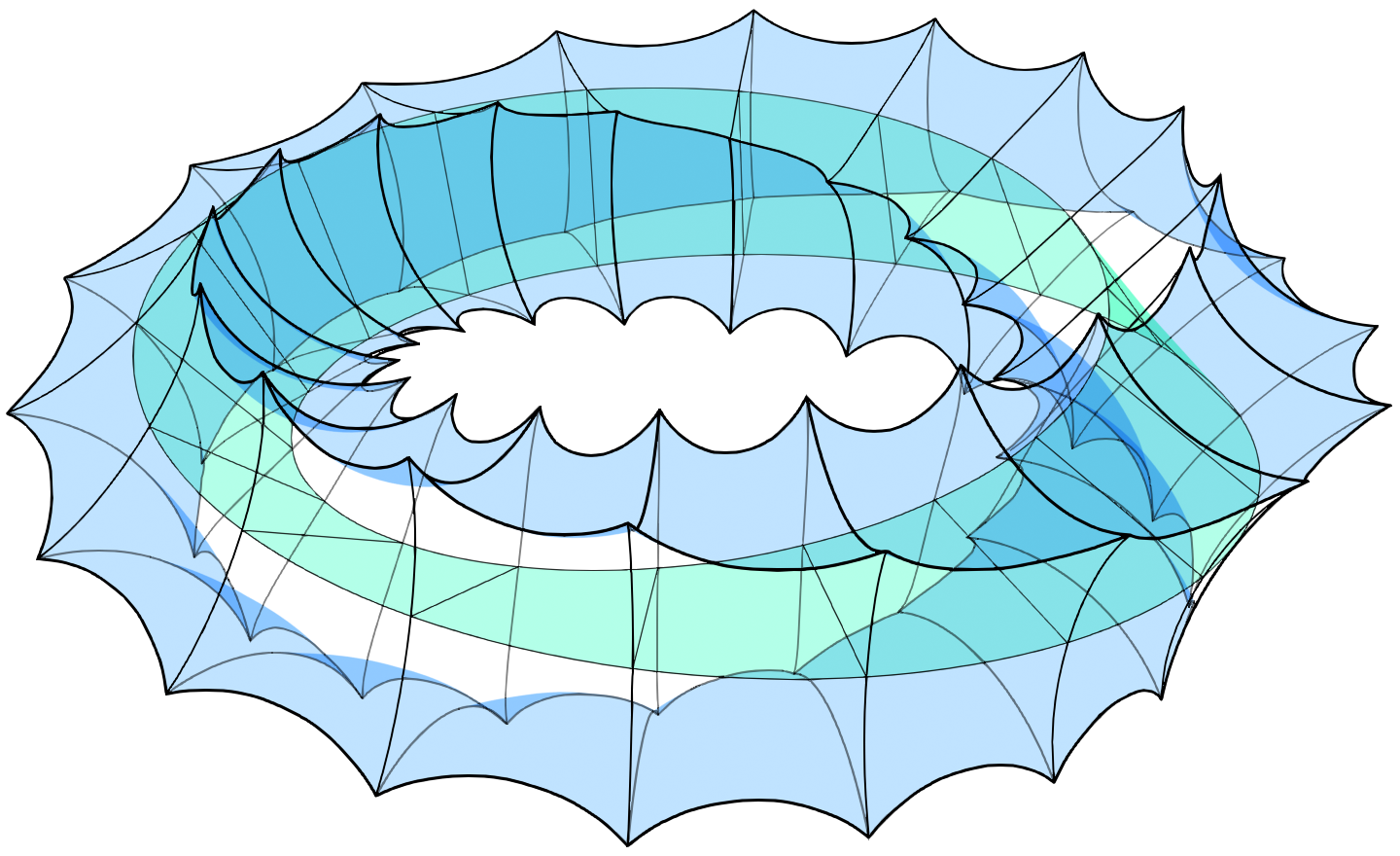}
    \caption{A three-dimensional representation of an orientable manifold with right-angled corners, obtained by arranging copies of $Q$ (here, a square prism) on a non-orientable hypersurface (here, a Möbius strip). Note that we will consider only $4$-manifolds without boundary, unlike the Möbius strip; hence, only the blue faces are relevant to our case.}
    \label{fig:mfd-corners}
\end{figure}

\section{Some \texorpdfstring{$4$}{4}-manifolds}\label{sec:4mfds}
The goal of this section is to find closed $4$-manifolds tessellated by copies of a $120$-cell $D$ with dihedral angle $2\pi/3$, which is equivalent to finding finite-index torsion-free subgroups of the Coxeter group
\begin{equation}
    G \coloneqq \raisebox{-0.6\height}{
    \begin{tikzpicture}
        \node[dot, fill=black, label={below:\strut $a$}] (a) at (0,0) {};
        \node[dot, fill=black, label={below:\strut $b$}] (b) at (1,0) {};
        \node[dot, fill=black, label={below:\strut $c$}] (c) at (2,0) {};
        \node[dot, fill=black, label={below:\strut $d$}] (d) at (3,0) {};
        \node[dot, fill=black, label={below:\strut $e$}] (e) at (4,0) {};
        \node[] (x) at (3,0.25) {};
        \node[] (x) at (3,-0.25) {};
        \draw[draw=black, double distance=3pt, thin, solid] (a) -- (b);
        \draw[draw=black, thin, solid] (a) -- (b);
        \draw[draw=black, thin, solid] (b) -- (c);
        \draw[draw=black, thin, solid] (c) -- (d);
        \draw[draw=black, thin, solid] (d) -- (e);
    \end{tikzpicture}
    }
    \label{eq:D-labeling}
\end{equation}
Some research has already been done on this problem, starting with the smallest known closed hyperbolic $4$-manifold, of Euler characteristic $8$, constructed by Conder-Maclachlan~\cite{conder-maclachlan}, which was later accompanied by more examples of Long~\cite{long} with the same volume. Our construction will take inspiration from the latter manifolds, which have a more canonical, algebraic definition.

\subsection{Long's manifolds}

Long's construction starts from noticing that $G$ has a subgroup
\begin{equation}
S \coloneqq \langle a, c, e, decd, bacbab \rangle
\end{equation}
of index $85$. The permutation action $\rho$ of $G$ on its right cosets defines a quotient of $G$ with image isomorphic to the simple group $\mathrm O(5,4)$, an orthogonal group over $\mathbb F_4$ of cardinality $\text{\num{979200}} = 2^8\cdot 3^2 \cdot 5^2 \cdot 17$; these claims and subsequent ones may be verified with a computer algebra system such as GAP~\cite{GAP4}.

In order to check that a subgroup has no torsion, it suffices to ensure that it does not intersect the conjugacy classes of prime-order elements of vertex stabilizers of $G$. Representatives of such classes were computed in~\cite{conder-maclachlan}.

The normal subgroup $K \coloneqq \ker \rho$ contains the prime torsion conjugacy class of $(abcd)^{15}$ and no others. Consider now a $17$-Sylow subgroup $C_{17} < \mathrm{O}(5,4)$: it turns out that the preimage $C_{17}K$ has the same property (and has index $\num{57600}$ in $G$). Finally, $8$ of the $15$ nontrivial maps in $\Hom(C_{17}K,\ZZ_2)$ do not annihilate $(abcd)^{15}$: equivalently, this torsion element is sent to a non-zero vector in $(C_{17}K)^\mathrm{ab} \otimes \ZZ_2$. Such maps define eight torsion-free subgroups $H_i$\ $(i = 1,\dots,8)$ of index \num{115200} in $G$, which correspond to manifolds $Z_i$ tessellated by \num{115200} simplices and $\num{115200}/\num{14400} = 8$ copies of the $120$-cell $D$.

As can be expected from the small size of these manifolds, their tessellations all involve self-adjacencies in copies of $D$, which makes them unsuitable for the coloring method (as defined in Section~\ref{sec:coloring}). Thus, we will look for larger manifolds: to this end, we consider the subgroups $K_i\coloneqq K \cap H_i$, which have index $2$ in $K$ and index $17$ in $H_i$. They are normalized by the whole $C_{17}K$, so they correspond to $17$-fold coverings of the Long manifolds with deck transformations of order $17$. Consequently, they are tessellated by $136 = 17\cdot 8$ copies of $D$. We will call these manifolds $X_i$.

The \emph{orientation subgroup} $G^+ \coloneqq \langle ab, bc, cd, de \rangle$ of index $2$, consisting of all words of even length, does not contain any of the $K_i$. It follows that the $X_i$ are non-orientable and that the subgroup $K_i^+ \coloneqq K_i \cap G^+$ corresponds to the orientable double cover of $X_i$, a manifold tessellated by $272$ copies of $D$.

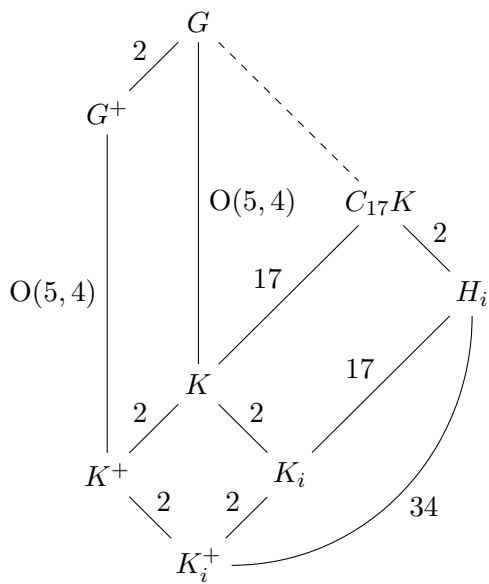
\begin{figure}
    \centering
    \begin{tikzpicture}[scale=1.2]
        \node (G)  at (0,0)    {$G$};
        \node (G2) at (-1,-1)  {$G^+$};
        \node (K)  at (0,-4)   {$K$};
        \node (CK) at (2,-2)   {$C_{17}K$};
        \node (Hi) at (3,-3)   {$H_i$};
        \node (Ki) at (1,-5)   {$K_i$};
        \node (K2) at (-1,-5)  {$K^+$};
        \node (Ki2) at (0,-6) {$K_i^+$};
        
        \draw[draw=black, thin, solid]
            (G) to node[above left=-2pt]{$2$} (G2);
        \draw[draw=black, thin, solid]
            (G) to node[right]{$\mathrm{O}(5,4)$} (K);
        \draw[draw=black, thin, solid]
            (G2) to node[left]{$\mathrm{O}(5,4)$} (K2);
        \draw[draw=black, thin, solid]
            (K) to node[above left=-2pt]{$2$} (K2);
        \draw[draw=black, thin, solid]
            (CK) to node[above left=-2pt]{$17$} (K);
        \draw[draw=black, thin, solid]
            (CK) to node[above right=-2pt]{$2$} (Hi);
        \draw[draw=black, thin, solid]
            (K) to node[above right=-2pt]{$2$} (Ki);
        \draw[draw=black, thin, solid]
            (K2) to node[above right=-2pt]{$2$} (Ki2);
        \draw[draw=black, thin, solid]
            (Hi) to node[above left=-2pt]{$17$} (Ki);
        \draw[draw=black, thin, solid]
            (Ki) to node[above left=-2pt]{$2$} (Ki2);
        \draw[draw=black, thin, solid]
            (Hi) to[out=270, in=0] node[below right=-2pt]{$34$} (Ki2);
        \draw[draw=black, thin, dashed]
            (G) -- (CK);
    \end{tikzpicture}
    \caption{Subgroup lattice of $G$. Solid lines denote normal subgroups, labeled by the quotient group or by an integer $n$, standing for $\ZZ_n$.}
    \label{fig:sg-lattice}
\end{figure}

We can now construct an orientable hyperbolic $5$-manifold with right-angled corners $Y_i$, whose interior is diffeomorphic to the determinant line bundle on $X_i$, by gluing copies of the prism $Q$ on both sides of each copy of $D$ in $X_i$. The \emph{facets} of $Y_i$, that is, its codimension-$1$ boundary components, are right-angled $120$-cells. In summary, we have:

\begin{thm}
    For each $i = 1,\dots,8$, there exists an orientable hyperbolic $5$-manifold $Y_i$ with right-angled corners, i.e. locally modeled on an orthant of $\HH^5$, whose $4$-dimensional facets are $272$ right-angled $120$-cells.
\end{thm}

\subsection{The adjacency graph}\label{sec:adj-graph}
In order to apply the coloring method to $Y_i$, we first need to check that none of its facets is glued to itself via some dodecahedron. Surprisingly, the facet adjacency graph is the same for all the $Y_i$ and can be described in a relatively simple way.
First, note that the graph also naturally describes adjacencies in the orientable double cover $X_i^+$ of $X_i$, tessellated by $272$ $120$-cells. Each $[5,3,3,3]$ simplex in $X_i^+$ corresponds to a coset $K_i^+ \gamma$ for some $\gamma \in G$, while each $120$-cell corresponds to a double coset $K_i^+ \gamma \Sigma$, where $\Sigma \coloneqq \langle a,b,c,d \rangle \simeq H_4$ is the symmetry group of the $120$-cell.

This subset is actually a left coset: since $K_i^+$ does not intersect the conjugacy class of $(abcd)^{15}$, we have $K^+ \coloneqq K \cap G^+ = K_i^+ \sqcup K_i^+ \cdot \gamma(abcd)^{15}\gamma^{-1}$ and
\begin{equation}
    K_i^+ \gamma \Sigma = K_i^+ \gamma [\Sigma \cup (abcd)^{15}\cdot \Sigma]
    = [K_i^+ \cup K_i^+\cdot  \gamma(abcd)^{15}\gamma^{-1}] \gamma\Sigma
    = K^+ \gamma\Sigma,
\end{equation}
which can be written as $\gamma \Sigma K^+$ by normality of $K^+$ in $G$. (Notice that this expression is already independent of $i$.) The vertices of our adjacency graph are left cosets of $\Sigma K^+$ in $G$. Passing to the quotient $G/K^+ \simeq \mathrm{O}(5,4) \times \ZZ_2$, these are in bijection with left cosets of $\overline{\Sigma}$, the image of $\Sigma$ in $G / K^+$. The map $\Sigma \to G / K^+$ has kernel of order $2$, so $|\overline{\Sigma}| = 7200$.

Regarding the edges, if we represent cells in the $120$-cell honeycomb as left cosets of $\Sigma$ in $G$, all pairs of adjacent cells are of the form $(\gamma\Sigma, \gamma e\Sigma)$, corresponding to the adjacency between the simplices $\gamma$ and $\gamma e = (\gamma e \gamma^{-1}) \gamma$. 

Summing up, the adjacency graph $\mathcal G$ we seek is given by:
\begin{align}
    V(\mathcal G) &= \{\gamma\overline{\Sigma} \mid \gamma \in G/K^+\},
\\  E(\mathcal G) &= \{(\gamma\overline{\Sigma}, \gamma e\overline{\Sigma}) \mid \gamma \in G/K^+\},
\end{align}
which is the same for all the $X_i^+$ (and the $Y_i$).

The graph $\mathcal G$ has no loops: indeed, it suffices that $\gamma \overline{\Sigma} \ne \gamma e \overline{\Sigma}$ for all $\gamma \in G/K^+$, which is equivalent to (the image of) $e$ not being an element of $\overline {\Sigma}$; we can easily check this using GAP. As a consequence, for all $i$, no facet of $Y_i$ meets itself along a ridge.

\begin{remark}
    The space $\HH^5 / K^+$ is an orbifold cover of the $[5,3,3,3]$ simplex, and is itself doubly covered by all the $X_i^+$. It is tessellated by $272$ copies of a ``half-$120$-cell'', obtained as a quotient of the standard $120$-cell by the action of $(abcd)^{15}$, which corresponds to multiplication by $-1$ in $\RR^4$. The graph $\mathcal G$ may also be described as the adjacency graph of this tessellation. 
\end{remark}

\begin{remark} \label{rmk:nicer-graph}
    There is a more convenient way to represent the set $V(\mathcal G)$. The right action of $G$ on the cosets of 
    $S \cap G^+ = \langle ac, ae, decd, bacbab \rangle$,
    of index $170$, has kernel $K^+$ and leads to a faithful permutation representation $G / K^+ \hookrightarrow S_{170}$. The orbits of $\overline{\Sigma}$ have sizes $150$, $10$, $10$; choose a $10$-element orbit $\omega$. It can be checked that $\overline{\Sigma}$, acting on the subsets of $\{1,2,\dots,170\}$, is the stabilizer of $\omega$. Therefore, $V(\mathcal G)$ can be constructed as the orbit of $\omega$ under the action of $G / K^+$ on subsets of $\{1,2,\dots,170\}$. 
\end{remark}
\section{Closed \texorpdfstring{$5$}{5}-manifolds}\label{sec:coloring}
In this section, we introduce the coloring method and apply it to our right-angled manifolds $Y_i$, obtaining some large closed hyperbolic $5$-manifolds with many symmetries, which we then use to obtain smaller manifolds as quotients.

\subsection{The coloring method}
The \emph{coloring method} is a well-known technique used to construct hyperbolic manifolds without boundary from right-angled polytopes~\cite{garrison-scott,davis-januszkiewicz,vesnin}.

Let $\mathcal P$ be a compact right-angled hyperbolic polytope with facets $\Fac(\mathcal P) = \{F_1, \dots, F_k\}$. More generally, we may allow $\mathcal P$ to be a right-angled manifold, by additionally requiring its facets to be \emph{embedded} or, equivalently, not self-adjacent along a ridge.

\begin{defin}
    A \emph{coloring} of $\mathcal P$ is a map $\lambda\colon \Fac(\mathcal P) \to V$ for some finite-dimensional vector space $V$ over $\ZZ_2$, assigning to each facet a \emph{color} in $V$, such that for each vertex of $\mathcal P$, the colors of its incident facets form a linearly independent set.
\end{defin}
A particularly simple class of colorings is obtained by considering only colors in a fixed basis of $V$. This is equivalent to coloring the facet adjacency graph of $\mathcal P$, which is the $1$-skeleton of the dual polyhedron, in such a way that no two vertices with the same color share an edge.
\begin{defin}
    Given a coloring $\lambda$, we can define the \emph{real toric manifold} $M(\mathcal P, \lambda)$ as a gluing of $|V|$ copies $\{\mathcal P_v\}$ of $\mathcal P$, indexed by vectors $v \in V$: for each facet $F \in \Fac(\mathcal P)$ and each vector $v \in V$, glue $\mathcal P_v$ and $\mathcal P_{v + \lambda(F)}$ along $F$, using the identity map of $F$.
\end{defin}
Due to the linear independence condition, the resulting space is indeed a closed manifold; moreover, it is connected if and only if $\image(\lambda)$ is a generating set for $V$. The construction is clearly invariant under vector space isomorphisms, so we may assume $V = \ZZ_2^m$ for some $m$. If we fix such an isomorphism, along with an ordering of $\Fac(\mathcal P)$, the coloring $\lambda$ can be summarized by a binary matrix whose columns are the colors of the facets of $\mathcal P$; this is called the \emph{characteristic matrix} of the coloring.

Multiplying the characteristic matrix by an invertible matrix on the left is equivalent to a vector space isomorphism, and does not change the isometry class of the real toric manifold. More generally, if $q\colon \ZZ_2^m \twoheadrightarrow \ZZ_2^n$ is a quotient of vector spaces such that $q\circ \lambda$ is also a coloring on $\mathcal P$, then $M(\mathcal P, \lambda)$ covers $M(\mathcal P, q\circ \lambda)$.

Now, let us define some colorings of the right-angled manifold $Y_i$ by taking advantage of its symmetries. If we inspect the diagram of Figure~\ref{fig:sg-lattice}, we see that a group of isometries $H_i / K_i^+$ isomorphic to $\ZZ_{17} \times \ZZ_2$ acts on both $X_i$ and its double cover, inducing isometries of $Y_i$. We will fix two generators that work for all $i$: $\psi = (abcde)^2 \cdot K_i^+$ of order $17$, and $\sigma$, the involution of $Y_i$ that fixes $X_i$ pointwise and acts by negation on its normal bundle.
Note that no cell of $Y_i$, of any dimension, can be fixed by $\psi^j, j=1,\dots,16$, because $17$ does not divide the order of any finite subgroup of $\Gamma$. Hence, $\psi$ acts freely on the graph $\mathcal G$, with $16$ orbits of size $17$.

Now, recall that an \emph{independent set} of a graph is a subset of its vertices inducing a subgraph with no edges. We call an independent set $I$ of $\mathcal G$ \emph{good} if it contains exactly one vertex from each orbit of $\psi$. Such a set with $\psi$ induces a partition of $\mathcal G$ into $17$ independent sets $(I, \psi(I), \dots, \psi^{16}(I))$; in turn, this defines a coloring $\lambda_{i,I}\colon \Fac(Y_i) \to \ZZ_2^{17}$, sending each facet in $\psi^j(I)$ to the basis vector $e_j$. We shall call the resulting manifold $M_{i,I} \coloneqq M(Y_i, \lambda_{i,I})$.

\subsection{Counting good independent sets}

All good independent sets of $\mathcal G$ can be obtained in SageMath by constructing all $16$-cliques of an auxiliary graph $\mathcal G^*$, obtained from $\mathcal G$ by adding a $17$-clique for each orbit of $\psi$ and then taking the complement. In this way, we find that there are exactly \num{13548660} good independent sets. 
However, this number does not take into account the symmetries of the manifolds $Y_i$. Let $\Pi_i$ be the group of isometries of $Y_i$ that preserve the natural tessellation by copies of $P$; we have $\Pi_i \simeq N_G(K_i) / K_i^+$. The order $|\Pi_i|$ divides $[G : K_i^+] = 2^{10} \cdot 3^2 \cdot 5^2 \cdot 17$; hence, $\langle \psi \rangle < \Pi_i$ is a $17$-Sylow subgroup and is unique up to conjugacy. 

%
%

Let $\Lambda_i$ be the normalizer of $\langle \psi \rangle$ in $\Pi_i$; any two good independent sets in the same orbit of $\Lambda_i$ produce identically colored $Y_i$, up to isometries of $Y_i$ and reordering of colors (which simply corresponds to a permutation of the canonical basis of $\ZZ_2^{17}$).
Using SageMath and GAP, we can count orbits of independent sets under the action of $\Lambda_i$ (Table~\ref{tab:orbits}).

\begin{table}
    \centering
    \begin{tabular}{ccr@{${}+{}$}lc}
         \toprule
         Groups & Size & \multicolumn{2}{c}{Orbit sizes} & Number of orbits \\
         \midrule[\heavyrulewidth]
         $\Lambda_1, \Lambda_3, \Lambda_6, \Lambda_8$
         & 136
         & $1618 \cdot [34]$ & $198436 \cdot [68]$ 
         & \num{200054} \\ \midrule
         $\Lambda_2, \Lambda_4, \Lambda_5, \Lambda_7$
         & 272
         & $809 \cdot [68]$ & $99218\cdot [136]$
         & \num{100027} \\
         \bottomrule
    \end{tabular}
    \caption{
    Sizes of orbits of the groups $\Lambda_i$, expressed as sums of terms of the form $n\cdot [m]$, meaning $n$ orbits of size $m$.
    }
    \label{tab:orbits}
\end{table}

Moreover, we also find that there are two well-defined \emph{types} of nontrivial elements of $\langle \psi \rangle$, preserved by the groups $\Lambda_i$:
\begin{defin}
    The \emph{type} of $\psi^k$, where $k \in \{\pm 1, \pm 2, \dots, \pm 8\}$, is $+$ if $k \in \{\pm 1, \pm 2, \pm 4, \pm 8\}$ (the \emph{squares} modulo $17$) and $-$ if $k \in \{\pm 3, \pm 5, \pm 6, \pm 7\}$ (the \emph{non-squares} modulo $17$).
\end{defin}
\begin{lemma}\label{lemma:type-welldef}
For every $i$, the conjugation action $\Lambda_i \curvearrowright \langle \psi \rangle \setminus \{\id\}$ preserves the type of an isometry.    
\end{lemma}
\begin{proof}
By Table~\ref{tab:orbits}, the order of $\Lambda_i$ divides $272 = 2 \cdot 136$. Now, notice that $\Lambda_i$ contains the order-$34$ subgroup $\langle \psi, \sigma \rangle$, which centralizes $\langle \psi \rangle$.
Hence, the conjugation action $\Lambda_i \to \operatorname{Aut}(\langle \psi \rangle) \simeq \ZZ_{17}^\times$ factors through the quotient $\Lambda_i / \langle \psi, \sigma \rangle$, whose order divides $272/34 = 8$.
It follows that its image is contained in the unique index-$2$ subgroup of $\operatorname{Aut}(\langle \psi \rangle) \simeq \ZZ_{16}$, whose orbits on $\langle \psi \rangle \setminus \{\id\}$ are exactly the two types.
\end{proof}

In conclusion, we obtain:
\begin{thm}
    There exists a family of closed hyperbolic $5$-manifolds $M_{i,I}$ tessellated by $2^{17}$ copies of $Y_i$ and by $2^{17} \cdot 272 \cdot 14400 = \num{513382809600}$ copies of the simplicial prism $P$.
\end{thm}
By abuse of notation, the subscript $I$ should be taken to mean the orbit of $I$ under $\Lambda_i$.

Using more general, vector-valued colorings, the number of prisms can be reduced by quite a lot; the main ingredient for this optimization is the theory of \emph{linear binary codes}.
\subsection{Linear binary codes}\label{sec:lin-codes}

A \emph{linear code} is a vector subspace $C$ of $\mathbb F_q^n$, where $\mathbb F_q$ is any finite field. We will only consider the case $q=2$, whence the adjective ``binary''. The ambient dimension $n$ is called the \emph{length} of the code. The main use case of codes is in the correction of errors in communication channels. By interpreting vectors as messages, two parties may exchange only elements of $C$, or \emph{codewords}, carrying $\dim C$ bits of information per codeword. It is desirable for any two distinct elements of $C$ to differ in many bits, say at least $d$. In this way, as many as $\left\lfloor \frac{d-1}{2} \right\rfloor$ bit-flip errors per codeword may be corrected by reverting to the nearest codeword. The maximum such $d$ is called the \emph{minimum distance} of $C$.

Consider the two binary matrices
\begin{equation}
A \coloneqq \begin{bmatrix}
1 & 1 & 1 & \zg{0} & 1 & \zg{0} & 1 & 1 & 1 & \zg{0} & \zg{0} & \zg{0} & \zg{0} & \zg{0} & \zg{0} & \zg{0} & \zg{0} \\
\zg{0} & 1 & 1 & 1 & \zg{0} & 1 & \zg{0} & 1 & 1 & 1 & \zg{0} & \zg{0} & \zg{0} & \zg{0} & \zg{0} & \zg{0} & \zg{0} \\
\zg{0} & \zg{0} & 1 & 1 & 1 & \zg{0} & 1 & \zg{0} & 1 & 1 & 1 & \zg{0} & \zg{0} & \zg{0} & \zg{0} & \zg{0} & \zg{0} \\
\zg{0} & \zg{0} & \zg{0} & 1 & 1 & 1 & \zg{0} & 1 & \zg{0} & 1 & 1 & 1 & \zg{0} & \zg{0} & \zg{0} & \zg{0} & \zg{0} \\
\zg{0} & \zg{0} & \zg{0} & \zg{0} & 1 & 1 & 1 & \zg{0} & 1 & \zg{0} & 1 & 1 & 1 & \zg{0} & \zg{0} & \zg{0} & \zg{0} \\
\zg{0} & \zg{0} & \zg{0} & \zg{0} & \zg{0} & 1 & 1 & 1 & \zg{0} & 1 & \zg{0} & 1 & 1 & 1 & \zg{0} & \zg{0} & \zg{0} \\
\zg{0} & \zg{0} & \zg{0} & \zg{0} & \zg{0} & \zg{0} & 1 & 1 & 1 & \zg{0} & 1 & \zg{0} & 1 & 1 & 1 & \zg{0} & \zg{0} \\
\zg{0} & \zg{0} & \zg{0} & \zg{0} & \zg{0} & \zg{0} & \zg{0} & 1 & 1 & 1 & \zg{0} & 1 & \zg{0} & 1 & 1 & 1 & \zg{0} \\
\zg{0} & \zg{0} & \zg{0} & \zg{0} & \zg{0} & \zg{0} & \zg{0} & \zg{0} & 1 & 1 & 1 & \zg{0} & 1 & \zg{0} & 1 & 1 & 1
\end{bmatrix},
\end{equation}
\begin{equation}
B \coloneqq \begin{bmatrix}
1 & 1 & \zg{0} & 1 & \zg{0} & \zg{0} & 1 & \zg{0} & 1 & 1 & \zg{0} & \zg{0} & \zg{0} & \zg{0} & \zg{0} & \zg{0} & \zg{0} \\
\zg{0} & 1 & 1 & \zg{0} & 1 & \zg{0} & \zg{0} & 1 & \zg{0} & 1 & 1 & \zg{0} & \zg{0} & \zg{0} & \zg{0} & \zg{0} & \zg{0} \\
\zg{0} & \zg{0} & 1 & 1 & \zg{0} & 1 & \zg{0} & \zg{0} & 1 & \zg{0} & 1 & 1 & \zg{0} & \zg{0} & \zg{0} & \zg{0} & \zg{0} \\
\zg{0} & \zg{0} & \zg{0} & 1 & 1 & \zg{0} & 1 & \zg{0} & \zg{0} & 1 & \zg{0} & 1 & 1 & \zg{0} & \zg{0} & \zg{0} & \zg{0} \\
\zg{0} & \zg{0} & \zg{0} & \zg{0} & 1 & 1 & \zg{0} & 1 & \zg{0} & \zg{0} & 1 & \zg{0} & 1 & 1 & \zg{0} & \zg{0} & \zg{0} \\
\zg{0} & \zg{0} & \zg{0} & \zg{0} & \zg{0} & 1 & 1 & \zg{0} & 1 & \zg{0} & \zg{0} & 1 & \zg{0} & 1 & 1 & \zg{0} & \zg{0} \\
\zg{0} & \zg{0} & \zg{0} & \zg{0} & \zg{0} & \zg{0} & 1 & 1 & \zg{0} & 1 & \zg{0} & \zg{0} & 1 & \zg{0} & 1 & 1 & \zg{0} \\
\zg{0} & \zg{0} & \zg{0} & \zg{0} & \zg{0} & \zg{0} & \zg{0} & 1 & 1 & \zg{0} & 1 & \zg{0} & \zg{0} & 1 & \zg{0} & 1 & 1
\end{bmatrix}.
\end{equation}
Their row spaces define linear binary codes $C(A), C(B)$ of length $17$ and dimensions $9$ and $8$, respectively, which are \emph{dual} to each other, meaning that $AB^t$ and $BA^t$ are zero matrices. Note that since $9+8=17$, we have $\image(B^t) = C(B) = \ker(A)$. These codes arise naturally as \emph{quadratic residue codes}, in a construction based on the fact that $2$, the order of the field, is a quadratic residue modulo $17$~\cite[481]{qr-codes}.

The minimum distance of $C(B)$ is known to be $6$, implying that no row vector with less than $6$ ones is in $C(B)$. Any relation of linear dependence between $k$ columns of $A$ gives a vector with $k$ ones in $\ker A = C(B)$. Hence, no $5$ columns of $A$ are linearly dependent.

The vertex figure of the order-$3$ $120$-cell honeycomb is a $4$-simplex, so cells of $X_i$ (and facets of $Y_i$) meet at most $5$ at a time. Therefore, if we take one of the $17$-colorings of $\mathcal G$ defined previously from an independent set $I$, and replace the basis vectors with the $17$ columns $A_0, A_1, \dots, A_{16}$ of $A$, we obtain a coloring over $\ZZ_2^9$. 

We can generalize this by introducing a permutation of the columns; specifically, given a good independent set $I$ and a generator $\psi^k \in \langle \psi \rangle$, $k \in \ZZ_{17}^\times$, define a coloring $\lambda_{i,I}^{(k)}$, respectively assigning colors $(A_0,A_1,\dots, A_{16})$ to $(I, \psi^{k}(I), \dots, \psi^{16k}(I))$. This results in a manifold $\widehat{M}_{i,I}^{(k)}$. Later on, the particular symmetry of these colorings will allow a further reduction in size.

\begin{prop}
    The isometry class of $\widehat{M}_{i,I}^{(k)}$ depends only on the type of $\psi^k$ and on the $\Lambda_i$-orbit of $I$.
\end{prop}
\begin{proof}
    Define a binary $9\times 17$ matrix $A^{(k)}$, whose $j$-th column $A^{(k)}_{j}$ is the color assigned to $\psi^{j}(I)$, that is, $A_{k^{-1}j}$ (columns indices and inverses are computed in $\ZZ_{17}$). For example, $A^{(1)} = A$. We say that $A^{(k)}$ and $A^{(k')}$ are \emph{left equivalent} if there exists a matrix $M \in \mathrm{GL}(9,2)$ such that $A^{(k)} = MA^{(k')}$; in that case, the characteristic matrices are also left equivalent and there is an isometry $\widehat{M}_{i,I}^{(k)} \simeq \widehat{M}_{i,I}^{(k')}$. This is a consequence of the fact that the coloring $\lambda_{i,I}^{(k)}$ assigns the color $A^{(k)}_j$ to each facet in $\psi^j(I)$, for all $j$.

    Using SageMath, we can check that the $A^{(k)}$ are divided into two left equivalence classes, corresponding exactly to the two types.
    Moreover, by Lemma~\ref{lemma:type-welldef}, acting on $I$ with $\Lambda_i$ has the same effect as replacing $\psi^k$ with another isometry $\psi^h$ of the same type; that is, it does not change the manifold.
\end{proof}

As a consequence, we may call these manifolds $\widehat{M}_{i,I}^\pm$. They are covered by the corresponding $M_{i,I}$, and are tessellated by $2^9$ copies of $Y_i$ and by $2^9 \cdot 272 \cdot 14400 = \num{2005401600}$ copies of the simplicial prism $P$.

\subsection{Even smaller manifolds}

We can still reduce this number: let us construct an extension of the order-$17$ isometry group $\langle \psi\rangle$ to the whole of $\widehat{M}_{i,I}^\pm$. For concreteness, we work with a specific $\widehat{M}_{i,I}^{(k)}$ and seek to extend $\psi^k$. The cyclic symmetry of the code $C(A)$ entails the existence of a matrix
\begin{equation}
R \coloneqq \begin{bmatrix}
1 & \zg{0} & 1 & \zg{0} & \zg{0} & 1 & \zg{0} & 1 & 1 \\
1 & \zg{0} & \zg{0} & \zg{0} & \zg{0} & \zg{0} & \zg{0} & \zg{0} & \zg{0} \\
\zg{0} & 1 & \zg{0} & \zg{0} & \zg{0} & \zg{0} & \zg{0} & \zg{0} & \zg{0} \\
\zg{0} & \zg{0} & 1 & \zg{0} & \zg{0} & \zg{0} & \zg{0} & \zg{0} & \zg{0} \\
\zg{0} & \zg{0} & \zg{0} & 1 & \zg{0} & \zg{0} & \zg{0} & \zg{0} & \zg{0} \\
\zg{0} & \zg{0} & \zg{0} & \zg{0} & 1 & \zg{0} & \zg{0} & \zg{0} & \zg{0} \\
\zg{0} & \zg{0} & \zg{0} & \zg{0} & \zg{0} & 1 & \zg{0} & \zg{0} & \zg{0} \\
\zg{0} & \zg{0} & \zg{0} & \zg{0} & \zg{0} & \zg{0} & 1 & \zg{0} & \zg{0} \\
\zg{0} & \zg{0} & \zg{0} & \zg{0} & \zg{0} & \zg{0} & \zg{0} & 1 & \zg{0}
\end{bmatrix}
\end{equation}
such that $R\cdot A_i = A_{i+1}$ for all $i \in \ZZ_{17}$. Therefore, if $v_j \in \ZZ_2^9$ is the color of facet $F_j$, then $Rv_j$ is the color of $\psi^k(F_j)$. 

We define our isometry $\Psi : \widehat{M}_{i,I}^{(k)} \to \widehat{M}_{i,I}^{(k)}$ as follows. If $x \in Y_i$, let $x_v$ be its corresponding
\\[-0.6ex]
point in a copy $Y_i^{(v)}, v \in \ZZ_2^9$, and let
\begin{equation}
    \Psi(x_v) \coloneqq \psi^k(x)_{Rv}.
\end{equation}
It is not hard to check that $\Psi$ is well-defined and gives indeed an isometry. For instance, if $x_v$ is on a boundary facet $F_j$ with color $v_j$, then it belongs to both $Y_i^{(v)}$ and $Y_i^{(v+v_j)}$. Since $\psi^k(x)$ is on a facet with color $Rv_j$, the points $\Psi(x_v) = \psi^k(x)_{Rv}$ and $\Psi(x_{v+v_j}) = \psi^k(x)_{Rv+Rv_j}$ are glued together, as required.

This isometry has order $17$ and has no fixed points; again, this is because no finite subgroup of $\Gamma$ has order a multiple of $17$. Moreover, it uniquely extends $\psi^k$ to $\widehat{M}_{i,I}^{(k)}$, so the action of $\langle \Psi \rangle$ depends only on the type of $\widehat{M}_{i,I}^{(k)}$. Thus, we can define quotients
\begin{equation}
    N_{i,I}^{\pm} \coloneqq \widehat{M}_{i,I}^{\pm} / \langle \Psi \rangle,
\end{equation}
which are tessellated by $2^9 \cdot 272 / 17 = 8192$ copies of $Q$ and $2^9 \cdot 272 \cdot 14400 / 17 = \num{117964800}$ copies of $P$. The volume of these manifolds, as we will see in the next section, is less than $\num{250000}$.

\subsection{Classification of the \texorpdfstring{$5$}{5}-manifolds}
We now discuss the classification of the manifolds $N_{i,I}^\pm$ up to isometry.
\begin{thm}
There are exactly $\num{1600432}$ pairwise non-isometric manifolds among the $N_{i,I}^\pm$, completely classified by an index $i \in \{1,4,5,2,3,7\}$, a $\Lambda_i$-orbit $I$ of good independent sets in $\mathcal G$, and a type in $\{ {+}, {-}\}$.
\end{thm}
The number $\num{1600432}$ can be obtained from the statement of the theorem and Table~\ref{tab:orbits}.
The proof is rather technical and heavily computer-assisted, and can be divided into three steps. First, we find that there are two pairs of Long manifolds related by tessellation-preserving isometries; then, we extend this result to the manifolds $N_{i,I}^\pm$, classifying them up to tessellation-preserving isometries; finally, we show that the Coxeter tessellation of $N_{i,I}^\pm$ is uniquely determined. The precise statements are as follows:
\begin{prop} \label{prop:long-isom}
Consider the Long manifolds $Z_i$, $i = 1,\dots,8$ with their natural tessellations into $[5,3,3,3]$ simplices. Then there is a tessellation-preserving isometry $Z_i\simeq Z_j$, $i<j$, if and only if $(i,j) \in \{(1,8), (3,6)\}$.
\end{prop}
\begin{prop} \label{prop:NiI-isom}
    Given a manifold $N_{i,I}^\pm$ with its natural tessellation by copies of $P$, we can uniquely recover:
    \begin{itemize}
        \item the Long manifold $Z_i$, up to tessellation-preserving isometry;
        \item the $\Lambda_i$-orbit of $I$;
        \item the type $\pm$.
    \end{itemize}
\end{prop}
\begin{prop}\label{prop:unique-tessellation}
    Every $N_{i,I}^\pm$ is uniquely tessellated by copies of $P$.
\end{prop}
The proofs of these three results can be found in the GitHub repository~\cite{git} as SageMath notebooks with accompanying explanations; even if they are mostly computational, it might still be interesting to summarize the main ideas.

In the case of Proposition~\ref{prop:long-isom}, we start by noticing that a tessellation-preserving isometry corresponds to conjugacy of the fundamental groups $\pi_1(Z_i)$ in the Coxeter group. Hence, proving the two isometries in the statement is a matter of finding an appropriate conjugating element in $G$; on the other hand, we find that the other Long manifolds $Z_i$ can be distinguished by their first homology groups (see~\cite[Table~1]{long}) or by counting quotients of $\pi_1(Z_i)$ isomorphic to the symmetric group $S_3$.

As for Proposition~\ref{prop:NiI-isom}, the idea is to consider the embedded submanifolds of $N_{i,I}^\pm$ built from facets corresponding to the generator $g$, of which there are $2$ isometric to $Z_i$ and $30$ isometric to $X_i$; the latter are found inside $30$ copies of $Y_i$. At this point, it remains to determine the type and the orbit of $I$: indeed, the adjacency relations between the copies of $Y_i$ allow us to define an auxiliary graph, and ultimately to extract the remaining information.

Finally, the proof of Proposition~\ref{prop:unique-tessellation} is by far the most technical of the three. The main idea is to classify hyperbolic elements of $\Gamma$ having a certain translation length, up to conjugacy. This is accomplished by searching the dual graph of the tessellation and finding conjugacy representatives via an algorithm based on~\cite{conjugacy-cox}. The associated geodesics can only occur in certain positions with respect to the tessellation. Eventually, by exploiting patterns in the arrangement of these geodesics, we recover a flag of hyperbolic subspaces of $\HH^5$, which determines the tessellation uniquely.

\section{Properties}\label{sec:properties}
In this section, we study various properties of the manifolds $N_{i,I}^\pm$ and related ones.
\subsection{Orientability}\label{sec:orientability}
\begin{prop}
    The manifolds $N_{i,I}^\pm$ are orientable.
\end{prop}

\begin{proof}
It is enough to prove orientability for $\widehat{M}_{i,I}^\pm$, since $\langle \Psi \rangle$ has odd order and thus preserves orientation. Recall that $\widehat{M}_{i,I}^\pm$ can be decomposed into $2^9$ copies of $Y_i$, which is orientable; hence, we just need to orient each copy consistently. Following the proof of~\cite[Lemma~2.4]{coloring-orient}, it suffices to find an isomorphism of $\ZZ_2^9$ that sends all facet colors (columns of $A$) to vectors with an odd number of ones.

If $u_n$ is the length-$n$ row vector of ones, this is equivalent to the existence of $M \in \mathrm{GL}(9,2)$ such that $u_9MA = u_{17}$. We can easily verify that, by taking
\begin{equation}
    w \coloneqq (1, 0, 0, 1, 1, 1, 0, 0, 1),
\end{equation}
we have $wA = u_{17}$. Hence, any $M$ such that $u_9 M = w$ works.
\end{proof}
\begin{remark}
The relative orientation of the copy of $Y_i$ indexed by $v \in \ZZ_2^9$ is determined by the value of $w\cdot v\in \ZZ_2$.
\end{remark}
\subsection{Volume}
In general, computing the volume of an odd-dimensional hyperbolic manifold is harder than in the even-dimensional case, as the Chern--Gauss--Bonnet formula is not available. However, there are number-theoretical techniques that apply to arithmetic orbifolds in all dimensions, based on \emph{Prasad's formula}~\cite{prasad}. In fact, an example of this is an expression for the volume of the simplicial prism $P$, essentially found in~\cite[Proposition~4]{vol-p}:
\begin{equation}
    \operatorname{vol}(P) =
    \frac{9\sqrt{5}^{15}}{32\pi^{15}}\zeta_{k_0}(2)\zeta_{k_0}(4)\zeta_{\ell_2}(3)/\zeta_{k_0}(3),
\end{equation}
where $k_0 = \QQ(\sqrt{5})$, $\ell_2 = \QQ(\sqrt{\varphi})$, $\varphi = \frac{\sqrt{5}+1}{2}$, and $\zeta_k$ is the Dirichlet zeta function for the field $k$. We can numerically evaluate this expression using PARI/GP~\cite{PARI2}, obtaining:
\begin{equation}
\begin{array}{@{} l @{} >{{}}l<{{}} @{} S[table-format=6.18] @{} l @{}}
    \operatorname{vol}(P) &=&          
         0.001984696430311649
    &{}\dots \\
    \operatorname{vol}(N_{i,I}^\pm) &=& 234124.317462427649199813
    &{}\dots
\end{array}
\end{equation}

\subsection{First homology}
As expected in high dimension, the manifolds we constructed are quite combinatorially complex, so that a direct computation of their homology, e.g.\ from a CW complex structure, seems elusive. However, some simple topological operations can reduce them to well-studied objects known as \emph{real toric spaces}, on which we can apply a formula of Suciu and Trevisan, generalized by Choi and Park~\cite{choi-park}.

\begin{defin}
    Given a simplicial complex $K$ on $m$ vertices, the \emph{real moment-angle complex} $\RR\mathcal Z_K$ of $K$ is defined as
    \begin{equation}
        \RR\mathcal Z_K \coloneqq \bigcup_{\sigma \in K}
        \{
            (x_1,\dots,x_m) \in [-1,1]^m \mid
            \text{$x_i \in \{\pm 1\}$ if $i \not \in \sigma$} 
        \}.
    \end{equation}
\end{defin}
These spaces are cubical subcomplexes of $[-1,1]^m$ and have a natural $\ZZ_2^m$-action by reflections along the coordinate axes.
\begin{defin}
    If $n < m$ and $\Lambda$ is a binary $n \times m$ matrix, then $\ker \Lambda \subseteq \ZZ_2^m$ acts on $\RR\mathcal Z_K$, and we define the \emph{real toric space} $M^\RR(K, \Lambda)$ as the quotient $\RR\mathcal Z_K / \ker \Lambda$.
\end{defin}

Now consider a manifold $M \coloneqq \widehat M^\pm_{i,I}$. It contains $2^9$ embedded copies of $X_i$, whose first homology $H_1(X_i) \simeq K_i^\mathrm{ab}$ is finite and has only $2$-torsion (see Table~\ref{tab:Xi-ab}).

\begin{table}[ht]
\centering
    \begin{tabular}{cc}
         \toprule Manifolds & First homology \\
         \midrule
         $X_1 \simeq X_8, X_4, X_5$
            & $\ZZ_2^{32} \oplus \ZZ_4^{11}$\\[1ex]
         $X_2, X_3 \simeq X_6$
            & $\ZZ_2^{27} \oplus \ZZ_4^{18}$ \\[1ex]
         $X_7$
            & $\ZZ_2^{27} \oplus \ZZ_4^{17} \oplus \ZZ_8$ \\
         \bottomrule
    \end{tabular}
    \caption{First homology groups of the manifolds $X_i, i = 1,\dots,8$.}
    \label{tab:Xi-ab}
\end{table}

If $F$ is a field with $\operatorname{char} F \ne 2$, then the homology exact sequence
\begin{equation}\label{eq:exact-collapse}
\begin{tikzcd}
0 = H_1(X_i; F) \arrow[r] & H_1(M; F) \arrow[r,"\sim"] & {H_1(M, X_i; F)} \arrow[r] & \widetilde H_0(X_i; F) = 0,
\end{tikzcd}
\end{equation}
implies $H_1(M; F) \simeq H_1(M, X_i; F) \simeq H_1(M/X_i; F)$. By continuing in this manner, we can collapse all copies of $X_i$ to points, while preserving $H_1$. The space $\overline M$ thus obtained can also be described as a gluing of $2^9$ copies of a cone over $\partial Y_i$.

Consider $\mathcal G$ as a $1$-dimensional simplicial complex; let $\lambda \coloneqq \lambda_{i,I}^{(k)}$ be the coloring used in the construction of $M$ and let $\Lambda$ be a $9\times 272$ characteristic matrix for $\lambda$. This enables us to define a real toric space $M_\mathcal G \coloneqq M^\RR(\mathcal G, \Lambda)$.

As one could expect, this has a connection with the real toric \emph{manifold} $M$, underlined by the following result.
\begin{prop}
We have $H_1(\overline M) \simeq H_1(M_\mathcal G)$ in homology with $F$-coefficients.
\end{prop}
\begin{proof} 
Note that $M$ is naturally tessellated by truncated $5$-cubes centered at the vertices of each $Y_i$, whose $4$-simplex facets lie on the $X_i$. Collapsing the latter creates a tessellation of $\overline M$ by $5$-cubes whose vertices are among the collapsed points.

We are only interested in the $2$-skeleton of this tessellation, which is composed of points, edges and squares. 
The points are in bijection with the copies of $Y_i$ in $M$, that is, with $\ZZ_2^9$. Edges correspond to facets of the $Y_i$ in $M$ (i.e.\ vertices of $\mathcal G$), while squares correspond to pairs of adjacent facets (edges of $\mathcal G$).

The construction perfectly mirrors that of $M_\mathcal G$, save for the fact that any two adjacent facets share two opposite dodecahedra (due to the symmetry $(abcd)^{15}$) and contribute to the $2$-skeleton with two squares having the same boundary. However, by excision, we may remove one of the two squares without changing the first homology.
\end{proof}
Now we shall apply the aforementioned formula, which we report here:
\begin{thm}[{\cite[Theorem~1.1]{choi-park}}]\label{thm:choi-park}
    Let $M = M^\RR(K, \Lambda)$ be a real toric space
and $R$ a commutative ring in which $2$ is a unit. Then there is an $R$-linear
isomorphism
\begin{equation}
    H^p(M; R) \simeq \bigoplus_{\omega \in \operatorname{row} \Lambda} \widetilde H^{p-1}(K_\omega; R),
\end{equation}
where $\operatorname{row} \Lambda$ is the row space of $\Lambda$, and $K_\omega \subseteq K$ is the subcomplex of $K$ induced by the support of $\omega$ as a subset of $\{1,\dots,m\}$.
\end{thm}
A direct consequence is
\begin{equation}\label{eq:sum-cc}
    b_1(M_\mathcal G; F) = \sum_{\omega \in (\operatorname{row} \Lambda) \setminus \{0\}} (b_0(\mathcal G_\omega; F) - 1).
\end{equation}
Heuristically, we expect the graphs $\mathcal G_\omega$ to be connected: since $5$ is the minimum distance of the code $C(A)$, every nonzero $\omega \in \operatorname{row} \Lambda$ has at least $80 = 5 \cdot 16$ ones. Using a custom program (written in C++ for performance reasons) we verify that this is the case for all choices of $i$, $I$ and a type $\pm$. It follows that the sum~(\ref{eq:sum-cc}) is always zero:
\begin{thm}
    Let $F$ be a field with $\operatorname{char} F \ne 2$. Then all manifolds $\widehat M_{i,I}^\pm$ have vanishing first homology with coefficients in $F$.
\end{thm}
\begin{cor}\label{cor:hom-m-hat}
    The integral homology group $H_1(\widehat M_{i,I}^\pm)$ is a finite abelian $2$-group.
\end{cor}
Recall that the manifold $\widehat M_{i,I}^\pm$ is a $17$-fold regular covering space of $N_{i,I}^\pm$. Using this fact, we can prove:
\begin{thm}
    The integral homology group $H_1(N_{i,I}^\pm)$ is of the form $T_2 \oplus \ZZ_{17}$, where $T_2$ is a finite abelian $2$-group obtained as a quotient of $H_1(\widehat M_{i,I}^\pm)$.
\end{thm}
\begin{proof}
Let $\widehat \pi \coloneqq \pi_1(\widehat M_{i,I}^\pm), \pi \coloneqq \pi_1(N_{i,I}^\pm)$; then $\widehat \pi$ is a normal subgroup of $\pi$ of index $17$. Moreover, by the Hurewicz theorem, we have $H_1(\widehat M_{i,I}^\pm) \simeq \widehat \pi / [\widehat \pi, \widehat \pi]$ and $H_1(N_{i,I}^\pm) \simeq \pi / [\pi, \pi]$. 

Since the quotient $\pi / \widehat \pi\simeq \ZZ_{17}$ is abelian, the commutator subgroup $[\pi, \pi]$ is contained (and is normal) in $\widehat \pi$. If we define $T_2 \coloneqq \widehat \pi / [\pi, \pi]$, which is a quotient of $H_1(\widehat M_{i,I}^\pm)$ (and hence a $2$-group by Corollary~\ref{cor:hom-m-hat}), we have $T_2 \triangleleft \pi / [\pi, \pi] = H_1(N_{i,I}^\pm)$. Finally, since $T_2$ and $H_1(N_{i,I}^\pm) / T_2 \simeq  \ZZ_{17}$ have coprime orders, we have $H_1(N_{i,I}^\pm) \simeq T_2 \oplus \ZZ_{17}$.
\end{proof}

\begin{cor}
    The manifold $\widehat M_{i,I}^\pm$ is the unique regular $17$-fold covering of $N_{i,I}^\pm$.
\end{cor}
\begin{remark}\label{rmk:retract-hom}
    More can be said about the structure of $H_1(\widehat M_{i,I}^\pm)$ and $T_2$. Since $\widehat M_{i,I}^\pm$ orbifold covers $Y_i$, which is injectively included in $\widehat M_{i,I}^\pm$, we have a retraction $\widehat M_{i,I}^\pm \twoheadrightarrow Y_i$. Composing with the natural retraction $Y_i \twoheadrightarrow X_i$ and taking homology, we get an injection from $H_1(X_i)$ (see Table~\ref{tab:Xi-ab}) to (the $2$-torsion of) $H_1(\widehat M_{i,I}^\pm)$. Similarly, by considering a $17$-fold quotiented $Y_i$ inside $N_{i,I}^\pm$, we get an injection $H_i(Z_i) \hookrightarrow T_2$.
\end{remark}
\begin{remark}\label{rmk:higher-betti}
    The computation of higher homology groups is much harder for several reasons, the first of which is the failure of  the collapsing trick~(\ref{eq:exact-collapse}), since $H_k(X_i; F)$ is not necessarily $0$. Of course, even overlooking this issue, we would have to consider more complicated real toric spaces, arising from higher-dimensional simplicial complexes; computing $b_1$ and higher for their subcomplexes would be more involved than a simple graph search, as in the case of $b_0$.
\end{remark}

\subsection{Parallelizability}
In this part, we discuss the problem of existence of a parallelizable hyperbolic $5$-manifold, and how it relates to the objects studied up to this point. To be more specific, we prove:
\begin{prop}\label{prop:virt-par-P}
    There exists a parallelizable closed hyperbolic $5$-manifold tessellated by copies of $P$.
\end{prop}
This is a direct consequence of a more general result, which may be of independent interest:
\begin{thm}\label{thm:virt-par}
    Let $n \in \{1,3,5,7\} \cup \{4k+1 \mid k \ge 2\}$. Then every closed hyperbolic $n$-manifold $M$ is virtually parallelizable.
\end{thm}
This allows further generalizations of Proposition~\ref{prop:virt-par-P}, for example involving manifolds tessellated by copies of $Y_i$.

The proof of Theorem~\ref{thm:virt-par} starts from a deep result of Deligne--Sullivan~\cite[553]{sullivan}, which states that every closed hyperbolic manifold is virtually stably parallelizable. Thus, let $M'$ be a stably parallelizable finite cover of $M$.
Assume first that $n \in \{1, 3, 7\}$. Then, by~\cite[652]{thomas}, $M'$ is also parallelizable and we are done.

In the remaining cases $n = 4k+1, k \ge 1$, again by~\cite[652]{thomas}, the only obstruction to parallelizability of $M'$ is the Kervaire semi-characteristic:
\begin{defin}
The \emph{Kervaire semi-characteristic} of a manifold $X$ is a $\ZZ_2$-valued invariant, defined by the formula
\begin{equation}
    \kappa(X) \coloneqq \sum_{i \ge 0} b_{2i}(X; \ZZ_2) \mod 2.
\end{equation}
\end{defin}

Suppose that $X$ is a $(4k+1)$-manifold with $w_{4k}(X) = 0$; then we can find two sections of $TM$ that are linearly independent at all but a finite number of points~\cite{thomas}, and there is an index formula for $\kappa$~\cite[Theorem~5.1]{atiyah-kervaire}. A small $4k$-sphere centered on a singular point naturally maps into the space of $2$-frames in $\RR^{4k+1}$, i.e.\ the Stiefel manifold $V_{4k+1,2}$. This defines an element in $\pi_{4k}(V_{4k+1,2}) \simeq \ZZ_2$, the \emph{index} of the singular point. Finally, the sum of all indices gives $\kappa(X)$.

A direct consequence of the index formula is:
\begin{prop}\label{prop:kervaire-even}
    If $X$ is a $(4k+1)$-manifold with $w_{4k}(X) = 0$, and $\overline X \twoheadrightarrow X$ is a cover of even degree $d$, then $\kappa(\overline X) = 0$.
\end{prop}
\begin{proof}
    Take a singular $2$-frame field on $X$ and lift it to $\overline X$. Every singular point of $X$ lifts to $d$ singular points with the same index, which cancel out. Hence, $\kappa(\overline X) = 0$.
\end{proof}

Since $M'$ is stably parallelizable, its Stiefel--Whitney classes vanish in each positive degree. Hence, by Proposition~\ref{prop:kervaire-even}, any even degree cover $M''$ of $M'$ is parallelizable, having $\kappa(M'') = 0$.
Finally, by~\cite[Theorem~A]{lubotzky}, every finitely generated linear group defined over $\RR$, such as $\pi_1(M')$, has a subgroup of even index, corresponding to a parallelizable cover $M'' \twoheadrightarrow M' \twoheadrightarrow M$.

This concludes the proof of Theorem~\ref{thm:virt-par}. Note that our proof is highly non-constructive, and hence we have no concrete representation of $M''$, nor a bound on its volume.

Finally, recall that even-dimensional closed hyperbolic manifolds cannot be parallelizable because of their Euler characteristic; thus it is natural to ask:
\begin{question}
    What can be said about virtual parallelizability of closed hyperbolic $(4k+3)$-manifolds, for $k\ge 2$?
\end{question}

\section{Another class of \texorpdfstring{$5$}{5}-manifolds}\label{sec:niceX}
The construction of small $5$-manifolds in Section~\ref{sec:coloring} involves a certain degree of asymmetry, due to the successive quotients and the choice of an independent set $I$; moreover, even the simplified definition of the graph $\mathcal G$ given in Remark~\ref{rmk:nicer-graph} falls short of providing a satisfactory description of the vertex set.

There is, however, a different right-angled $5$-manifold with a more elegant definition. The construction starts with a $4$-manifold tessellated by $650$ order-$3$ $120$-cells, obtained with Everitt and Maclachlan's method~\cite{everitt-maclachlan}. The general idea is to start with a linear representation of the group $G = \langle a,b,c,d,e \rangle$ over some integral domain, in our case $\ZZ[\varphi]$, and then pass to a finite field quotient; in most cases (including ours), the kernel of the quotient map gives a torsion-free finite-index subgroup and therefore a manifold.

\subsection{Linear representations}
Recall that every Coxeter group admits a canonical linear representation, which can be constructed from its Coxeter diagram.
First, we define a bilinear form $B_{ij}$ on $\RR^n$, where $n$ is the number of nodes: if $m_{ij}$ is the label of the edge between nodes $i$ and $j$, let
\begin{equation}
    B_{ij} \coloneqq -\cos(\pi / m_{ij}),
\end{equation}
with the conventions $m_{ii} = 1$ and $\pi/\infty = 0$. Then, the $i$-th generator $g_i$ is realized as the reflection in the $i$-th coordinate vector:
\begin{equation}
    g_i \mapsto [x \mapsto x - 2B(x,e_i)e_i].
\end{equation}
By~\cite[Chapter~5, §4.4, Corollary~2]{bcox}, this representation is faithful, so we will identify group elements and matrices. Note that in general this representation differs from that obtained by Vinberg's algorithm in the arithmetic case.

In the case of $G$, the images of the generators end up in $\mathrm{GL}(5,\ZZ[2\cos(\pi/5)]) = \mathrm{GL}(5,\ZZ[\varphi])$.
By~\cite[Corollary~1]{everitt-maclachlan}, the natural quotient map $q\colon \mathrm{GL}(5,\ZZ[\varphi]) \twoheadrightarrow \mathrm{GL}(5,\mathbb F_5)$, induced by the ideal $(\varphi - 3)$, has torsion-free kernel. Let $\alpha,\beta,\gamma,\delta,\varepsilon, K$ be the images of $a,b,c,d,e,B$ under $q$. After a change of basis that diagonalizes $K$, computed with SageMath, we have:
\begin{align}
    \alpha &=
    \begin{bmatrix}
        4 & \zg0 & \zg0 & \zg0 & \zg0 \\
        \zg0 & 1 & \zg0 & \zg0 & \zg0 \\
        \zg0 & \zg0 & 1 & \zg0 & \zg0 \\
        \zg0 & \zg0 & \zg0 & 1 & \zg0 \\
        \zg0 & \zg0 & \zg0 & \zg0 & 1
    \end{bmatrix}\!,\hspace{-2.5em}&
    \beta &=
    \begin{bmatrix}
        4 & 3 & 1 & \zg0 & \zg0 \\
        3 & 4 & 1 & \zg0 & \zg0 \\
        1 & 1 & 3 & \zg0 & \zg0 \\
        \zg0 & \zg0 & \zg0 & 1 & \zg0 \\
        \zg0 & \zg0 & \zg0 & \zg0 & 1
    \end{bmatrix}\!,\hspace{-2.5em}&
    \gamma &=
    \begin{bmatrix}
        1 & \zg0 & \zg0 & \zg0 & \zg0 \\
        \zg0 & 1 & \zg0 & \zg0 & \zg0 \\
        \zg0 & \zg0 & 4 & \zg0 & \zg0 \\
        \zg0 & \zg0 & \zg0 & 1 & \zg0 \\
        \zg0 & \zg0 & \zg0 & \zg0 & 1
    \end{bmatrix}\!,\\
    \delta &=
    \begin{bmatrix}
        1 & \zg0 & \zg0 & \zg0 & \zg0 \\
        \zg0 & 4 & 1 & 3 & \zg0 \\
        \zg0 & 1 & 3 & 1 & \zg0 \\
        \zg0 & 3 & 1 & 4 & \zg0 \\
        \zg0 & \zg0 & \zg0 & \zg0 & 1 \\
    \end{bmatrix}\!,\hspace{-2.5em}&
    \varepsilon &=
    \begin{bmatrix}
        1 & \zg0 & \zg0 & \zg0 & \zg0 \\
        \zg0 & 1 & \zg0 & \zg0 & \zg0 \\
        \zg0 & \zg0 & 1 & \zg0 & \zg0 \\
        \zg0 & \zg0 & \zg0 & 3 & 1 \\
        \zg0 & \zg0 & \zg0 & 2 & 2
    \end{bmatrix}\!,\hspace{-2.5em}&
    K &=
    \begin{bmatrix}
        1 & \zg0 & \zg0 & \zg0 & \zg0 \\
        \zg0 & 1 & \zg0 & \zg0 & \zg0 \\
        \zg0 & \zg0 & 1 & \zg0 & \zg0 \\
        \zg0 & \zg0 & \zg0 & 1 & \zg0 \\
        \zg0 & \zg0 & \zg0 & \zg0 & 3
    \end{bmatrix}\!.
\end{align}
The group $Q \coloneqq \image q = \langle \alpha, \beta, \gamma, \delta, \varepsilon \rangle$ has cardinality \num{9360000}, so it has index $2$ in the whole orthogonal group $\mathrm{O}(5,5)$ (which does not depend on the choice of a bilinear form). It can be characterized as the kernel of the \emph{spinor norm} $\mathrm{O}(5,5)\to \mathbb F_5^\times / (\mathbb F_5^\times)^2$ (as defined in~\cite[178]{cassels}); its group structure may also be described as the unique nontrivial semidirect product of the simple group $\mathrm{S}_4(5)$ with $\ZZ_2$, or equivalently as $\operatorname{Aut}(\mathrm{S}_4(5))$.

\subsection{The \texorpdfstring{$4$}{4}-manifold}
Now we are ready to describe the manifold $X$ associated to $\ker q$. The $120$-cell stabilizer $\Sigma = \langle a,b,c,d \rangle$ injects into $Q$, and the largest subspace of $\mathbb{F}_5^5$ that it preserves is the span of $v_0\coloneqq (0,0,0,0,1)$. Since the stabilizer of $v_0$ is exactly $q(\Sigma)$, there is a bijection between the orbit $Qv$ and the set of $120$-cells that comprise $X$. The adjacency structure between the $120$-cells can be obtained just like in Section~\ref{sec:adj-graph}, by taking the orbit of the pair of adjacent cells $(v_0, \varepsilon v_0)$. The resulting graph $\mathcal G'$ has $650$ vertices, each of degree $120$, and \num{78000} edges, and it admits a beautiful description:
\begin{align}
    V(\mathcal G') &= \{v \in \mathbb F_5^5 \mid K(v,v) = 3\},
\\  E(\mathcal G') &= \{(v,w) \in \mathbb F_5^5 \times \mathbb F_5^5 \mid K(v,w) = 1\}.
\end{align}
\begin{remark}
    Interestingly, this alternate definition is reminiscent of hyperbolic $4$-space, but as a smooth variety over a finite field: indeed, $K$ is a bilinear form with a single non-square diagonal entry, like the Lorentzian form on $\RR^{4,1}$, and by interpreting it as a distance, the edges of $\mathcal G'$ appear to describe points that are close together, in some sense.
\end{remark}
We could ask if $\mathcal G'$ contains enough information to recover the topology of $X$, since it is naturally embedded in the $1$-skeleton of the dual tessellation. However, the latter has \num{78000} $4$-simplices, while the induced flag simplicial complex of $\mathcal G'$ has exactly twice as many $5$-cliques, which we call \emph{real} if they correspond to a $4$-simplex and \emph{virtual} otherwise.

We can check that $\operatorname{Aut}(\mathcal G')$ acts transitively on the set of $5$-cliques, while $Q$ partitions it into two orbits of size \num{78000}. It is not hard to see that the action of $Q$ must send real cliques to real cliques, so we can arbitrarily declare one $Q$-orbit, such as that of $\{v_0,
\varepsilon v_0,
\delta \varepsilon v_0,
\gamma \delta \varepsilon v_0,
\beta \gamma \delta \varepsilon v_0\}$, as the set of real cliques, and construct $X$ as a simplicial complex.

Now we can use the GAP package HAP~\cite{HAP} to compute the integral homology of $X$:
\begin{equation}
    H_0(X) = H_4(X) = \ZZ,\enspace
    H_1(X) = H_3(X) = \ZZ^{144},\enspace
    H_2(X) = \ZZ^{936}.
\end{equation}
The Euler characteristic is $650$: this is to be expected, since each $120$-cell has $\chi = 1$.

As before, in order to construct a $5$-manifold, we start with a manifold with corners $Y$, obtained by arranging $650$ copies of the usual $120$-cell prism onto one side of $X$. Indeed, $X$ being orientable, this produces an orientable $Y$ with $X$ as a totally geodesic boundary component. In the coloring method, this facet can be arbitrarily colored, as it does not meet any other facet. The other $650$ facets are isometric to the right-angled $120$-cell, and we will call them \emph{small}. Their adjacency graph is isomorphic to $\mathcal G'$.

\begin{remark}
    If $F$ is a small facet of $Y$ corresponding to $v \in V(\mathcal G')$, then there is a natural map $\operatorname{Stab}_Q(v) \to \operatorname{Isom}(F)$, which is injective by rigidity of isometries. It can be checked that both domain and codomain have cardinality ${14400}$ (in particular, $\operatorname{Isom}(F)$ is the Coxeter group $[5,3,3]$). Hence, the map is also surjective: every isometry of a small facet extends to an isometry of $Y$.
\end{remark}

\subsection{Colorings}
One could expect the number $13$ to play a crucial role in coloring $\mathcal G'$, if only because of the prime factorization $650 = 2 \cdot 5^2 \cdot 13$. In fact, if $\lambda_1, \lambda_n$ are the largest and smallest eigenvalues of $\mathcal G'$, at least $1-\lambda_1/\lambda_n = 1 - 120/(-10) = 13$ colors are needed by the \emph{Hoffman bound}~\cite{hoffman}. As we will see in a moment, this bound is realized.

Unlike the case of the manifolds $N_{i,I}^\pm$, there is no ``good'' independent set with one element in each orbit of a $13$-Sylow subgroup. However, there is a different, reasonably symmetrical construction. Up to conjugacy, there is exactly one subgroup $L < Q$ whose action on $\mathcal G'$ has $26$ orbits of size $25$, which are all independent sets. It has order $125$, and it is generated by the two matrices
\begin{equation}
    \begin{bmatrix}
           0 &    3 &    2 &    0 &    2 \\
           4 &    0 &    1 &    4 &    2 \\
           1 &    1 &    0 &    1 &    3 \\
           4 &    2 &    3 &    2 &    2 \\
           4 &    3 &    2 &    0 &    3 \\
    \end{bmatrix}\!,\enspace
    \begin{bmatrix}
           1 &    2 &    0 &    2 &    1 \\
           3 &    2 &    2 &    1 &    3 \\
           0 &    3 &    1 &    3 &    4 \\
           3 &    1 &    2 &    2 &    3 \\
           3 &    1 &    2 &    1 &    4 \\
    \end{bmatrix}\!.
\end{equation}
If a smaller coloring is desired, these orbits can be paired together into $13$ independent sets of size $50$, in $64$ different ways. At the cost of introducing this choice, if we color each independent set with a basis vector of $\ZZ_2^{13}$, we obtain manifolds tessellated into $2^{13}$ copies of $Y$.

\begin{remark}\label{rmk:small-coloring-650}
    An even more efficient yet less symmetrical coloring, using binary vectors, arises from the dual of a linear binary code of length $13$, dimension $4$ and minimum distance $6$~\cite{codes}. Its columns are elements of $\ZZ_2^9$, corresponding to $512$ copies of $Y$ and $650 \cdot \num{14400} \cdot 512 = \num{4792320000}$ copies of $P$, for a volume of $\num{9511300.396911}\dots$.
\end{remark}

\setlength\bibitemsep{0.7ex}
\printbibliography[heading=bibintoc, title={References}]

@manual{sagemath,
  Key          = {SageMath},
  Author       = {{The Sage Developers}},
  Title        = {{S}ageMath, the {S}age {M}athematics {S}oftware {S}ystem ({V}ersion 10.5)},
  url          = {https://www.sagemath.org},
  Year         = {2024},
}

@manual{GAP4,
    key          = "GAP",
    organization = "The GAP~Group",
    title        = "{GAP -- Groups, Algorithms, and Programming,
                    Version 4.14.0}",
    year         = 2024,
    url          = "https://www.gap-system.org",
}

@manual{HAP,
    key          = "HAP",
    organization = "G. Ellis",
    title        = "{HAP -- Homological Algebra Programming,
                    Version 1.66}",
    year         = 2024,
    url          = "http://www.gap-system.org/Packages/hap.html",
}

@article{conder-maclachlan,
 ISSN = {00029939, 10886826},
 author = {M. Conder and C. Maclachlan},
 journal = {Proceedings of the American Mathematical Society},
 number = {8},
 pages = {2469--2476},
 publisher = {American Mathematical Society},
 title = {Compact Hyperbolic 4-Manifolds of Small Volume},
% urldate = {2024-01-14},
 volume = {133},
 year = {2005}
}

@article{long,
author = {Long, C.},
title = {Small volume closed hyperbolic 4-manifolds},
journal = {Bulletin of the London Mathematical Society},
volume = {40},
number = {5},
pages = {913-916},
doi = {https://doi.org/10.1112/blms/bdn077},
year = {2008}
}

@article{coloring-orient,
 ISSN = {00029939, 10886826},
 author = {A. Kolpakov and B. Martelli and S. Tschantz},
 journal = {Proc. Amer. Math. Soc.},
 number = {9},
 pages = {4103--4111},
 publisher = {American Mathematical Society},
 title = {Some hyperbolic three-manifolds that bound geometrically},
 volume = {143},
 year = {2015}
}

@article{prasad,
author = {Prasad, G.},
journal = {Publications Mathématiques de l'IHÉS},
keywords = {global field; semi-simple algebraic group; volume; S-arithmetic subgroup; reductive groups; bound for class numbers},
pages = {91-114},
publisher = {Institut des Hautes Études Scientifiques},
title = {Volumes of S-arithmetic quotients of semi-simple groups},
volume = {69},
year = {1989},
%language = {eng},
}

@misc{codes,

    title        = {Bounds on the minimum distance of linear codes},
    author       = {M. Grassl},
    year         = 2009,
    note         = {\url{
http://codetables.de/BKLC/} (Visited 2024-02-14)}
}

@inbook{everitt-maclachlan, 
    place={Cambridge},
    series={London Mathematical Society Lecture Note Series}, 
    title={Constructing hyperbolic manifolds}, 
    booktitle={Computational and Geometric Aspects of Modern Algebra}, 
    publisher={Cambridge University Press}, 
    author={Everitt, B. and Maclachlan, C.}, 
    year={2000}, 
    pages={78–86}, 
    collection={London Mathematical Society Lecture Note Series}
}

@book{bcox,
  title={Lie Groups and Lie Algebras: Chapters 4--6},
  note={{Translated} from {French} by {Andrew} {Pressley}},
  author={Bourbaki, N.},
  year={2008},
  publisher={Springer Berlin Heidelberg}
}

@incollection{cassels,
title = {Rational Quadratic Forms},
booktitle = {North-Holland Mathematics Studies},
publisher = {North-Holland},
volume = {74},
pages = {9-26},
year = {1982},
doi = {https://doi.org/10.1016/S0304-0208(08)70410-9},
author = {J. W. S. Cassels},
}

@inbook{hoffman,
author = {A. J. Hoffman},
title = {On Eigenvalues and Colorings of Graphs},
booktitle = {Selected Papers of Alan J. Hoffman},
chapter = {},
pages = {407-419},
doi = {10.1142/9789812796936_0041},
}

@book{qr-codes,
title = {The Theory of Error-Correcting Codes},
author = {F. J. MacWilliams and N. J. A. Sloane},
series = {North-Holland Mathematical Library},
publisher = {Elsevier},
year = {1977},
issn = {0924-6509},
doi = {https://doi.org/10.1016/S0924-6509(08)70541-5},
}

@article{choi-park,
author = {Choi, S. and Park, H.},
year = {2019},
month = {11},
pages = {},
title = {Multiplication structure of the cohomology ring of real toric spaces},
volume = {22},
journal = {Homology, Homotopy and Applications},
doi = {10.4310/HHA.2020.v22.n1.a7}
}

@book{atiyah-kervaire,
author="Atiyah, M. F.",
title="Vector Fields on Manifolds",
year="1970",
publisher="VS Verlag f{\"u}r Sozialwissenschaften",
isbn="978-3-322-98503-3",
doi="10.1007/978-3-322-98503-3_1",
}

@article{thomas,
  title={Vector fields on manifolds},
  author={E. Thomas},
  journal={Bulletin of the American Mathematical Society},
  year={1969},
  volume={75},
  pages={643-683},
}

@incollection{sullivan,
title = {Hyperbolic geometry and homeomorphisms},
booktitle = {Geometric Topology},
publisher = {Academic Press},
pages = {543-555},
year = {1979},
isbn = {978-0-12-158860-1},
doi = {https://doi.org/10.1016/B978-0-12-158860-1.50034-4},
author = {D. Sullivan}
}

@article{lubotzky,
    author = {Lubotzky, A.},
    title = "{On Finite Index Subgroups of Linear Groups}",
    journal = {Bulletin of the London Mathematical Society},
    volume = {19},
    number = {4},
    pages = {325-328},
    year = {1987},
    month = {07},
    issn = {0024-6093},
    doi = {10.1112/blms/19.4.325},
}

@article{vol-p,
author = {V. Emery and R. Kellerhals},
title = {{The three smallest compact arithmetic hyperbolic $5$–orbifolds}},
volume = {13},
journal = {Algebraic \& Geometric Topology},
number = {2},
publisher = {MSP},
pages = {817 -- 829},
year = {2013},
doi = {10.2140/agt.2013.13.817},
}

@manual{PARI2,
  organization = "{The PARI~Group}",
  title        = "{PARI/GP version \texttt{2.17.0}}",
  year         = 2024,
  address      = "Univ. Bordeaux",
  note         = "available from \url{http://pari.math.u-bordeaux.fr/}"
}

@article{conjugacy-cox,
    author = {D. Krammer},
    title = {The conjugacy problem for Coxeter groups},
    volume = {1},
    journal = {Groups Geom. Dyn.},
    number = {3},
    year = {2009},
    pages = {71-–171}
}

@article{bugaenko,
    title = {Groups of automorphisms of unimodular hyperbolic quadratic forms over the ring $\mathbf{Z}\bigl[\frac{\sqrt5+1}2\bigr]$},
    author = {V. O. Bugaenko},
    year = {1984},
    issue = {5},
    pages = {6--12},
    journal = {Vestnik Moskov. Univ. Ser.~1. Mat. Mekh.},
    language = {(In Russian)},
}

@article{garrison-scott,
    author = {A. Garrison and R. Scott},
    year = {2003},
    pages = {963--971},
    title = {Small covers of the dodecahedron and the 120-cell},
    volume = {131},
    journal = {Proc. Amer. Math. Soc.},
    doi = {10.1090/S0002-9939-02-06577-2}
}

@article{PV,
    author = {L. Potyagailo and E. Vinberg},
    title = {On right-angled reflection groups in hyperbolic spaces},
    journal = {Comment. Math. Helv.},
    volume = {80},
    date = {2005},
    number = {1},
    pages = {63--73}
}

@article{davis-januszkiewicz,
  author    = {M. Davis and T. Januszkiewicz},
  title     = {Convex polytopes, {C}oxeter orbifolds and torus actions},
  journal   = {Duke Mathematical Journal},
  volume    = {62},
  pages     = {417--451},
  year      = {1991},
}

@article{vesnin,
  author    = {A. Yu. Vesnin},
  title     = {Three-dimensional hyperbolic manifolds of {L\"obell} type},
  journal   = {Siberian Mathematical Journal},
  volume    = {28},
  pages     = {731--734},
  year      = {1987},
}

@misc{git,
    author = {{G}itHub repository},
    sortname = {GitHub},
    date = {2025},
    url = {https://github.com/floatingpoint-754/closed-hyp-5-manifolds}
}
\end{document}